\documentclass[reqno]{amsart}
\usepackage{hyperref} 

\begin{document}
\title[Nodal solutions of fourth-order Kirchhoff equations]
{Nodal solutions of fourth-order Kirchhoff equations with critical growth in $\mathbb{R}^N$}

\author[H. Pu, S. Li, S. Liang,  D. D. Repov\v{s}]
{Hongling Pu, Shiqi Li, Sihua Liang, Du\v{s}an D. Repov\v{s}}

\address{Hongling Pu \newline
College of Mathematics,
Changchun Normal University,
Changchun 130032, China}
\email{pauline\_phl@163.com}

\address{Shiqi Li \newline
College of Mathematics,
Changchun Normal University,
Changchun 130032,  China}
\email{lishiqi59@126.com}

\address{Sihua Liang \newline
College of Mathematics,
Changchun Normal University,
Changchun 130032,  China}
\email{liangsihua@163.com}

\address{Du\v{s}an D. Repov\v{s}\newline
Faculty of Education and Faculty of Mathematics and Physics,
University of Ljubljana \& Institute of Mathematics, Physics and Mechanics,
Ljubljana, 1000, Slovenia}
\email{dusan.repovs@guest.arnes.si}

\keywords{Fourth-order elliptic equation; Kirchhoff problem; critical exponent;
\hfill\break\indent variational methods; nodal solution}

\begin{abstract}
 We consider a class of  fourth-order elliptic equations of
 Kirchhoff type with critical growth in $\mathbb{R}^N$.
 By using  constrained minimization in the Nehari manifold, we
 establish sufficient conditions for the existence  of nodal
 (that is, sign-changing) solutions.
\end{abstract}

\maketitle
\numberwithin{equation}{section}
\newtheorem{theorem}{Theorem}[section]
\newtheorem{lemma}[theorem]{Lemma}
\newtheorem{definition}[theorem]{Definition}
\allowdisplaybreaks

\section{Introduction}

In this article we studies the existence of nodal solutions to the fourth-order
elliptic equations of Kirchhoff type with critical growth in $\mathbb{R}^N$,
\begin{equation}\label{e1.1}
\Delta^2 u-\Big(1+b\int_{\mathbb{R}^N} |\nabla u|^2\,dx\Big)\Delta u+V(x)u
=\lambda f(u)+|u|^{2^{**}-2}u, \quad x\in\mathbb{R}^N,
\end{equation}
where $\Delta^2 u$ is the biharmonic operator, $2^{**}=2N/(N-4)$ is
the critical Sobolev exponent with $5 \leq N < 8$, and $b$ and $\lambda$
are  positive parameters. The continuous functions $V(x)$ and
$f(u)$ satisfy the following conditions:
\begin{itemize}
\item[(A1)] $V\in C(\mathbb{R}^N,\mathbb{R})$ satisfies
$\underset{x\in\mathbb{R}^N}{\inf}V(x)\ge V_0>0$, where $V_0$ is a positive
constant. For each $M > 0$, $\operatorname{meas}\{x \in  \mathbb{R}^N: V(x) \leq
M\} < \infty$, where $\operatorname{meas}(\cdot)$ denotes the Lebesgue measure on
$\mathbb{R}^N$;

\item[(A2)] $f \in C^1(\mathbb{R},\mathbb{R})$ and $f(u)=o(|u|)$, as $u\to 0$;

\item[(A3)] There exists $p\in(4, 2^{**})$ such that $\lim_{u\to\infty}f(u)/u^{p-1}=0$;

\item[(A4)] $\lim_{u\to\infty}F(u)/u^4=+\infty$, where $F(u)=\int_0^u f(t)dt$;

\item[(A5)] $f(u)/|u|^3$ is a strictly increasing function for
$u\in\mathbb{R}\setminus\{0\}$.
\end{itemize}

Problem \eqref{e1.1} originates from the  Kirchhoff equation
\begin{equation}\label{e1.2}
-\Big(a+b\int_\Omega |\nabla u|^2\,dx\Big)\Delta u=f(x,u)\quad\text{in }\Omega,
\end{equation}
where $\Omega\in\mathbb{R}^N$ is a bounded domain, $a>0$, $b\ge0$, and $u$ satisfies certain
boundary conditions. The above equation stems from a typical model proposed by
Kirchhoff \cite{ki},
\begin{equation}\label{e1.3}
u_{tt}-\Big(a+b\int_\Omega |\nabla u|^2\,dx\Big)\Delta u=f(x,u),
\end{equation}
which serves as a generalization of the classical D'Alembert wave equation
\begin{eqnarray*}
\rho\frac{\partial^2u}{\partial t^2}-\Big(\frac{\rho_0}{h} +
\frac{E}{2L}\int_0^L|\frac{\partial u}{\partial
x}|^2dx\Big)\frac{\partial^2u}{\partial x^2} = f(x,u),
\end{eqnarray*}
by taking into account the effects of changes in the length of
strings during vibrations. The nonlocal term thus appears. See for
example \cite{car1, do} for more background on such problems. Thanks
to the pioneering work of Lions \cite{lions} on problem
\eqref{e1.3}, a lot of attention has been drawn to these nonlocal
problems during the last decade. That was followed by some
interesting results on the existence of various solutions to
\eqref{e1.2}, including positive solutions, multiple solutions,
bound state solutions, multibump solutions, and semiclassical state
solutions, both on bounded domains and on the entire space. For more
results on the Kirchhoff-type equations we  refer to \cite{ch, LS,
liang1, LPZ, MRS,  sun} and the references therein. Problem
\eqref{e1.2} with critical nonlinearity, however, is seldom
covered, mainly because of the challenge - the lack of compactness -
presented by the presence of the critical Sobolev exponent. We also
refer the interested readers to \cite{fis2, liang3,  pu3, x2, x3,
z3} on the fractional  Kirchhoff type problems.

Recently, various approaches have been adopted for considering the  fourth-order elliptic
equations of the Kirchhoff type,
\begin{gather*}
\Delta^2 u-\Big(a+b\int_\Omega |\nabla u|^2\,dx\Big)\Delta u
 =f(x,u),\quad  x\in\Omega,\\
u=\Delta u=0,\quad  x\in\partial\Omega,
\end{gather*}
where $\Delta^2 u$ is the biharmonic operator, with different
hypotheses on the nonlinearity. For instance, Ma \cite{ma} studied
the existence and multiplicity of positive solutions to the
fourth-order equation with the fixed point theorems in cones of
ordered Banach spaces. Wang et al. \cite{wang} applied the mountain
pass and the truncation methods to get the existence of nontrivial
solutions to the fourth-order elliptic equations of the Kirchhoff type with
one parameter $\lambda$. Liang and Zhang \cite{liang2} used the
variational methods to obtain the existence and multiplicity of
solutions to the fourth-order elliptic equations of the Kirchhoff type with
critical growth in $\mathbb{R}^N$.

The motivation for this paper comes from \cite{sh, ta1, wa3, z2}. In
\cite{sh}, the existence was proved of one least energy nodal
solution $u_b$ to problem \eqref{e1.2}, with its energy strictly
larger than the ground state energy. Meanwhile, the asymptotic
behavior of $u_b$, as the parameter $b\searrow 0$, was investigated as
well. Later, under some more weak assumptions on $f$ (especially,
with the Nehari type monotonicity condition removed), Tang and Cheng
\cite{ta1} improved and generalized some results obtained in
\cite{sh} with some new analytical skills and the non-Nehari manifold
method. In \cite{wa3}, the authors obtained the existence of least
energy nodal solutions to the Kirchhoff-type equation with critical
growth in bounded domains by using the constraint variational method
and the quantitative deformation lemma. In \cite{z2}, the authors
studied the  fourth-order elliptic equation of the
Kirchhoff-type,
\begin{gather*}
\Delta^2 u-\Big(a+b\int_{\mathbb{R}^N} |\nabla u|^2\,dx\Big)\Delta u+V(x)u
= f(u), \quad x\in\mathbb{R}^N,\\
u \in H^2(\mathbb{R}^N),
\end{gather*}
where $a > 0$ and $b \geq 0$ are constants. By the constraint
variational method and the quantitative deformation lemma, they proved
that the problem possesses one least energy nodal solution. For more
results on nodal solutions to the Kirchhoff-type equations, please refer
to \cite{de, lif, lu, z1} and the references therein. However, to
the best of our knowledge, there are no such results concerning
the existence of nodal solutions of the problem \eqref{e1.1}
involving critical nonlinearities in the whole space.

The purpose of this article is to study  the existence, energy
estimates and convergence properties of the least energy nodal
solutions to the fourth-order elliptic equation \eqref{e1.1}.
The novelty of this paper is that  problem \eqref{e1.1} concerns the critical
case on the entire space. Based on these facts, the problem turns out to be extremely
complicated and more difficult than the one without critical
nonlinearities in bounded domains. Since problem \eqref{e1.1}
involves critical exponents in the nonlinearity, it is rather
difficult to show that the energy functional reaches a lower infimum
on the Nehari manifold  because of the lack of compactness caused by the
critical term. As we will see, this problem prevents us from
using the approach in \cite{a2, sh, ta1, z2}. So we need some new
ideas to  overcome the above difficulties. Moreover, we use the
constraint variational method, the topological degree theory and the
quantitative deformation lemma to prove our main results.  Thus, our
main results generalize papers \cite{a2, sh, ta1, z2} in several
directions.

Before stating our main results, we define
\[
H^2(\mathbb{R}^N)
:=\{ u\in L^2(\mathbb{R}^N): |\nabla u|, \Delta u\in L^2(\mathbb{R}^N)\},
\]
endowed with the norm
\[
\| u\|_{H^2(\mathbb{R}^N)}
=\Big( \int_{\mathbb{R}^N} \big( (\Delta u)^2+(\nabla u)^2+u^2\big)\,dx \Big)^{1/2}.
\]
Now, we introduce the space
\[
E:=\big\{ u\in H^2(\mathbb{R}^N): \int_{\mathbb{R}^N}V(x)|u|^2\,dx<\infty \big\}
\]
with the inner product
\[
\langle u,v\rangle
=\int_{\mathbb{R}^N} \left( \Delta u\Delta v+\nabla u\nabla v+V(x)uv \right)\,dx
\]
and the norm
\[
\| u\| =\int_{\mathbb{R}^N} \left( |\Delta u|^2+|\nabla
u|^2+V(x)|u|^2 \right)\,dx.
\]
Under condition (A1), it is known that the embedding
$E\hookrightarrow H^2(\mathbb{R}^N)\hookrightarrow L^p(\mathbb{R}^N)$ for $p\in
(2,2^{**})$ is compact, and
continuous for $p\in [2,2^{**}]$ (see \cite{bar2}), and
\begin{equation}\label{e1.4}
S_p|u|_p\leq\|u\|,\quad \text{for every }  u\in E.
\end{equation}
In particular,  the best Sobolev constant for the embedding
$E\hookrightarrow L^{2^{**}}(\mathbb{R}^N)$ is
\[
S=\inf\big\{\int_{\mathbb{R}^N}|\Delta u|^2dx: \int_{\mathbb{R}^N}|u|^{2^{**}}dx=1 \big\}.
\]

\begin{definition} \rm
We say that $u\in E$ is a weak solution to problem \eqref{e1.1}, if
\begin{align*}
&\int_{\mathbb{R}^N}\left( \Delta u\Delta\phi+\nabla u\nabla\phi+V(x)u\phi
\right)\,dx +b \int_{\mathbb{R}^N}|\nabla u|^2dx \int_{\mathbb{R}^N}\nabla u\nabla\phi
\,dx\\
&=\lambda\int_{\mathbb{R}^N}f(u)\phi\,dx +\int_{\mathbb{R}^N}|u|^{2^{**}-2}u\phi\,dx,
\end{align*}
for every $\phi\in E$.
\end{definition}

For convenience, we will omit the term weak throughout this
paper. The corresponding energy functional $I_b^\lambda : E\to\mathbb{R}$ to
problem \eqref{e1.1} is defined by
\begin{equation} \label{e1.5}
\begin{aligned}
I_b^\lambda(u)
= &\frac{1}{2}\int_{\mathbb{R}^N}\big(|\Delta
u|^2+|\nabla u|^2+V(x)|u|^2 \big)\,dx +\frac{b}{4}\Big(
\int_{\mathbb{R}^N}|\nabla u|^2\,dx \Big)^2
\\
&-\lambda\int_{\mathbb{R}^N}F(u)\,dx
-\frac{1}{2^{**}}\int_{\mathbb{R}^N}|u|^{2^{**}}dx.
\end{aligned}
\end{equation}
It is easy to see that $I_b^\lambda$ belongs to $C^1(E,\mathbb{R})$ and the
critical points of $I_b^\lambda$ are the solutions to \eqref{e1.1}.
For every $u\in E$ we can write
\[
u^+(x)=\max\{u(x),0\}\quad \text{and} \quad u^-(x)=\min\{u(x),0\}\,.
\]
Then every solution $u\in E$ to problem \eqref{e1.1}
with the property that $u^{\pm}\neq 0$ is a nodal solution to
problem \eqref{e1.1}.

Our objective  is to find the least energy nodal solutions to
problem \eqref{e1.1}. There exist several interesting
studies on the following typical semilinear equation, which is
related to problem \eqref{e1.1} (see \cite{bar1, bar2}),
\begin{equation}\label{e1.8}
-\Delta u+V(x)u=f(x,u) \quad \text{in } \mathbb{R}^N.
\end{equation}
These methods, however, depend heavily upon the decompositions:
\begin{gather}\label{e1.9}
J(u)=J(u^+)+J(u^-), \\
\label{e1.10}
\langle J'(u),u^+\rangle=\langle J'(u^+),u^+\rangle \quad \text{and} \quad
\langle J'(u),u^-\rangle=\langle J'(u^-),u^-\rangle,
\end{gather}
where $J$ is the energy functional of \eqref{e1.8}, given by
\[
J(u) = \frac{1}{2}\int_{\mathbb{R}^N}(|\nabla
u|^2+V(x)u^2)\,dx-\int_{\mathbb{R}^N}F(x,u)\,dx.
\]
However, if $b>0$, the energy functional $I_b^{\lambda}$ cannot be
decomposed in the same way as it is done in \eqref{e1.9} and
\eqref{e1.10}. In fact, we have
\begin{align*}
 I_b^{\lambda}(u) = I_b^{\lambda}(u^{+}) + I_b^{\lambda}(u^{-}) +
 \frac{b}{2}\int_{\mathbb{R}^N}|\nabla u^+|^2dx\int_{\mathbb{R}^N}|\nabla u^-|^2dx;
\end{align*}
if $u^+ \not\equiv 0$, then
\begin{align*}
\langle (I_b^{\lambda})'(u),u^{+}\rangle = \langle
(I_b^{\lambda})'(u^+),u^{+}\rangle + b\int_{\mathbb{R}^N}|\nabla
u^+|^2dx\int_{\mathbb{R}^N}|\nabla u^-|^2dx > \langle
(I_b^{\lambda})'(u^+),u^+\rangle;
\end{align*}
if $u^- \not\equiv 0$, then
\begin{align*}
\langle (I_b^{\lambda})'(u),u^{-}\rangle = \langle
(I_b^{\lambda})'(u^-),u^{-}\rangle + b\int_{\mathbb{R}^N}|\nabla
u^-|^2dx\int_{\mathbb{R}^N}|\nabla u^+|^2dx > \langle
(I_b^{\lambda})'(u^-),u^-\rangle.
\end{align*}
 Therefore, the methods used for obtaining nodal solutions to the
local problem \eqref{e1.8} do not seem applicable to problem
\eqref{e1.1}. In this paper, we follow the approach in \cite{bar} by
defining the constrained set
\begin{equation}\label{e1.11}
\mathcal{N}_b^\lambda =\{u\in E: u^\pm\neq 0, \langle
(I_b^\lambda)'(u),u^\pm\rangle=0\}
\end{equation}
and considering a minimization problem of $I_b^\lambda$ on
$\mathcal{N}_b^\lambda$.
 Shuai \cite{sh} proved that $\mathcal{N}_b^\lambda \neq \emptyset$,
in the absence of the nonlocal term, by applying the
parametric method and the implicit theorem. However, it is the nonlocal
terms in problem \eqref{e1.1}, the biharmonic operator and the
nonlocal term involved, that add to our difficulties. Roughly
speaking, compared to the general Kirchhoff type problem
\eqref{e1.2}, decompositions \eqref{e1.9} and \eqref{e1.10}
corresponding to $I_b^\lambda$, are much more complicated, which
accounts for some technical difficulties during the proof of the
nonemptiness of $\mathcal{N}_b^\lambda$. Moreover, the parametric method and
implicit theorem are not applicable to problem \eqref{e1.1} because
the complexity of the nonlocal problem there. Hence, inspired by
\cite{al}, we follow a different path, specifically, we resort to a
modified Miranda's theorem (see \cite{mi}). It is also feasible to
prove that the minimizer of the constrained problem is also a nodal
solution via the quantitative deformation lemma and  degree theory.
We can now present our first main result.

\begin{theorem}\label{th1.1}
Assume that {\rm (A1)---(A5)} hold. Then there exists $\lambda^* > 0$
such that for all $\lambda \geq  \lambda^*$,  problem \eqref{e1.1} has a least energy
nodal solution  $u_b\in\mathcal{N}_b^\lambda$ such that
$I_b^\lambda(u_b)=\inf_{u\in\mathcal{N}_b^\lambda} I_b^\lambda(u)$.
\end{theorem}

Another goal of this paper is to establish the so-called energy
doubling property (cf. \cite{we}), i.e., the energy of any nodal
solution to problem \eqref{e1.1} is strictly larger than twice the
ground state energy. The conclusion is trivial for the semilinear
equation problem \eqref{e1.8}.   When $b > 0$, a similar result was
obtained by Shuai \cite{sh} in a bounded domain $\Omega$. We are
also interested in whether energy doubling property still holds
for problem \eqref{e1.1}. To answer this question, we prove the
following result.

\begin{theorem}\label{th1.2}
Assume that {\rm(A1)--(A5)} hold. Then there exists $\lambda^{**} > 0$ such that for all
$\lambda \geq \lambda^{**}$,  $c^* := \inf_{u \in \mathcal{M}_b^\lambda} I_b^\lambda(u) > 0$
is achieved, and $I_b^\lambda(u)> 2c^*$, where
$\mathcal {M}_b^\lambda = \{u \in E\setminus\{0\}: \langle (I_b^\lambda)'(u), u \rangle = 0 \}$
and $u$ is the least energy nodal solution obtained in Theorem \ref{th1.1}. In
particular, $c^* > 0$ is achieved either by a positive or a negative function.
\end{theorem}

It is obvious that the energy of the nodal solution $u_b$ obtained
in Theorem \ref{th1.1} depends on $b$. Next, we establish a
convergence property of $u_b$ as $b \to 0$, which demonstrates a
relationship between $b > 0$ and $b = 0$ for problem \eqref{e1.1}.

\begin{theorem}\label{th1.3}
Assume that {\rm (A1)--(A5)} hold.
Then for any sequence $\{b_n\}$ with $b_n \to 0$ as $n \to \infty$, there exists a
subsequence, still denoted by $\{b_n\}$, such that $\{u_n\}$ strongly converges to $u_0$ in $E$
as $n\to \infty$, where $u_0$ is a least energy nodal solution to
the problem
\begin{equation}\label{e1.14}
\Delta^2 u -\Delta u+V(x) = \lambda f(u)+|u|^{2^{**}-2}u \quad\text{in } \mathbb{R}^N.
\end{equation}
\end{theorem}

The structure of this article is as follows:
Section 2 contains the proof of the achieving the least
energy for the constraint problem \eqref{e1.1}.
While section 3 is devoted to the proofs of our main theorems.

Throughout this paper, we use standard notation. For simplicity,
we use ``$\to$'' and ``$\rightharpoonup$'' to denote the strong and weak convergence
in the related function space, respectively.
By $C$ and $ C_{i}$ we denote various positive constants, and by ``$:=$'' definitions.
To simplify the notation, we denote a subsequence of a sequence $\{u_n \}_n$ also as
$\{u_n \}_n$, unless otherwise  specified.

\section{Some technical lemmas}

To begin, fix $u \in E$ with $u^{\pm} \neq 0$. Consider the  function
$\varphi: \mathbb{R}_+\times \mathbb{R}_+\to\mathbb{R}$ and the
mapping $W: \mathbb{R}_+\times
\mathbb{R}_+\to\mathbb{R}^2$, where
\begin{gather}\label{e2.1}
\varphi(\alpha,\beta)=I_b^\lambda(\alpha u^++\beta u^-), \\
\label{e2.2}
W(\alpha,\beta)=\left(\langle(I_b^\lambda)'(\alpha u^++\beta u^-),\alpha u^+\rangle,
\langle(I_b^\lambda)'(\alpha u^++\beta u^-),\beta u^-\rangle\right).
\end{gather}
For brevity, we define the quantities
\[
A^+(u)=\int_{\mathbb{R}^N} |\nabla u^+|^2\,dx,
\quad
A^-(u)=\int_{\mathbb{R}^N} |\nabla u^-|^2\,dx,
\quad
B(u)=\int_{\mathbb{R}^N}\Delta u^+\Delta u^-\,dx.
\]

\begin{lemma}\label{lem2.1}
Assume that {\rm (A1)--(A5)} hold. Then for any $u \in E$ with
$u^{\pm} \neq 0$, there is the unique maximum point pair of positive
numbers $(\alpha_u, \beta_u)$ such that $\alpha_uu^++\beta_uu^-\in
\mathcal{N}_b^\lambda$.
\end{lemma}

\begin{proof}
Our proof consists in verifying three claims.
\smallskip

\noindent\textbf{Claim 1.}
There exists a pair of positive numbers  $(\alpha_u,\beta_u)$ such that
$\alpha_uu^++\beta_uu^-\in \mathcal{N}_b^\lambda$, for any
$u \in E$ with $u^{\pm} \neq 0$.
Note that
\begin{align*}
&\langle(I_b^\lambda)'(\alpha u^++\beta u^-),\alpha u^+\rangle\\
&= \int_{\mathbb{R}^N}\Delta(\alpha u^++\beta u^-)\Delta \alpha u^+dx
+\int_{\mathbb{R}^N}|\nabla\alpha u^+|^2\,dx
+\int_{\mathbb{R}^N}V(x)|\alpha u^+|^2dx
\\
&\quad
+b\int_{\mathbb{R}^N}|\nabla(\alpha u^++\beta u^-)|^2dx
\int_{\mathbb{R}^N}|\nabla\alpha u^+|^2dx \\
&\quad -\lambda\int_{\mathbb{R}^N} f(\alpha u^+)\alpha u^+dx
-\int_{\mathbb{R}^N}|\alpha u^+|^{2^{**}}dx
\end{align*}
and
\begin{align*}
&\langle(I_b^\lambda)'(\alpha u^++\beta u^-),\beta u^-\rangle \\
&=\int_{\mathbb{R}^N}\Delta(\alpha u^++\beta u^-)\Delta \beta u^-dx
+\int_{\mathbb{R}^N}|\nabla\beta u^-|^2\,dx
+\int_{\mathbb{R}^N}V(x)|\beta u^-|^2dx
\\
&\quad
+b\int_{\mathbb{R}^N}|\nabla(\alpha u^++\beta u^-)|^2dx
\int_{\mathbb{R}^N}|\nabla\beta u^-|^2dx
\\
&\quad -\lambda\int_{\mathbb{R}^N} f(\beta u^-)\beta u^-dx
-\int_{\mathbb{R}^N}|\beta u^-|^{2^{**}}dx.
\end{align*}
By a direct computation we obtain that
\begin{equation} \label{e2.3}
\begin{aligned}
&\langle(I_b^\lambda)'(\alpha u^++\beta u^-),\alpha u^+\rangle\\
&= \alpha^2\|u^+\|^2
+\alpha^2\beta^2 b A^+(u) A^-(u)
+\alpha^4 b\left(A^+(u)\right)^2\\
&\quad +\alpha\beta B(u)
-\lambda\int_{\mathbb{R}^N} f(\alpha u^+)\alpha u^+dx
-\int_{\mathbb{R}^N}|\alpha u^+|^{2^{**}}dx
\end{aligned}
\end{equation}
and
\begin{equation} \label{e2.4}
\begin{aligned}
&\langle(I_b^\lambda)'(\alpha u^++\beta u^-),\beta u^-\rangle \\
&=\beta^2\|u^-\|^2+\alpha^2\beta^2 b A^+(u) A^-(u)
+\beta^4 b\left(A^-(u)\right)^2
\\
&\quad +\alpha\beta B(u)
-\lambda\int_{\mathbb{R}^N} f(\beta u^-)\beta u^-dx
-\int_{\mathbb{R}^N}|\beta u^-|^{2^{**}}dx.
\end{aligned}
\end{equation}
By assumptions (A2) and (A3), we have
\begin{equation}\label{e2.5}
\int_{\mathbb{R}^N}f(\alpha u^+)\alpha u^+dx \leq \varepsilon
\int_{\mathbb{R}^N}|\alpha u^+|^2\,dx + C_\varepsilon \int_{\mathbb{R}^N}|\alpha
u^+|^p\,dx.
\end{equation}
Choose $\varepsilon>0$ small enough such that $(1-\lambda\varepsilon
C_\varepsilon)>0$, which together with \eqref{e2.5} and \eqref{e2.3},
yields
\begin{align*}
\langle(I_b^\lambda)'(\alpha u^++\beta u^-),\alpha u^+\rangle
&\ge (1-\lambda\varepsilon C_\varepsilon)\alpha^2\|u^+\|^2
+\alpha^2\beta^2 b A^+(u) A^-(u) \\
&\quad
+\alpha^4 b\left(A^+(u)\right)^2
-\lambda C_\varepsilon \int_{\mathbb{R}^N}|\alpha u^+|^p\,dx
-\int_{\mathbb{R}^N}|\alpha u^+|^{2^{**}}dx.
\end{align*}
Since $2^{**}>4$, we have $\langle(I_b^\lambda)'(\alpha u^++\beta u^-),\alpha u^+\rangle>0$
for a small enough $\alpha$  and for all $\beta\ge0$.

Similarly, according to \eqref{e2.5} and \eqref{e2.4}, we get
$\langle(I_b^\lambda)'(\alpha u^++\beta u^-),\beta u^-\rangle>0$,
for small enough $\beta$  and all $\alpha\ge0$.
Hence, there exists $r>0$ such that
\begin{equation}\label{e2.6}
\langle(I_b^\lambda)'(ru^++\beta u^-),ru^+\rangle>0\quad \text{and} \quad
\langle(I_b^\lambda)'(\alpha u^++ru^-),ru^-\rangle>0,
\end{equation}
for all $\alpha, \beta\ge0$.

On the other hand,  by (A3) and (A4), we have
\begin{align}\label{e2.7}
f(t)t > 0, t \neq 0;\quad F(t) \geq 0,\quad t \in \mathbb{R}.
\end{align}
Now, choose $R>r$. For  sufficiently large  $R$, and by \eqref{e2.3},
\eqref{e2.4}, \eqref{e2.7}, we have
\begin{align}\label{e2.8}
\langle(I_b^\lambda)'(Ru^++\beta u^-),Ru^+\rangle<0\quad \text{and} \quad
\langle(I_b^\lambda)'(\alpha u^++Ru^-),Ru^-\rangle<0,
\end{align}
for all $\alpha,\beta\in [r,R]$.
Invoking Miranda's theorem \cite{mi}, together with \eqref{e2.6}
and \eqref{e2.8}, we can conclude that there exists
$(\alpha_u,\beta_u)\in \mathbb{R}_+ \times
\mathbb{R}_+$ such that $W(\alpha_u,\beta_u)=(0,0)$, i.e.,
$\alpha_uu^++\beta_uu^-\in \mathcal{N}_b^\lambda$.
\smallskip

\noindent\textbf{Claim 2.}
The pair $(\alpha_u,\beta_u)$ is unique.

\noindent $\bullet$ Case $u \in \mathcal{N}_b^\lambda$. Then we have
\[
\langle(I_b^\lambda)'(u),u^+\rangle=0\quad \text{and} \quad
\langle(I_b^\lambda)'(u),u^-\rangle=0,
\]
that is,
\begin{equation} \label{e2.9}
\begin{aligned}
& \|u^+\|^2+B(u)+bA^+(u)\big( A^+(u)+A^-(u) \big)\\
& =\lambda\int_{\mathbb{R}^N}f(u^+)u^+dx +\int_{\mathbb{R}^N}|u^+|^{2^{**}}dx
\end{aligned}
\end{equation}
and
\begin{equation} \label{e2.10}
\begin{aligned}
& \|u^-\|^2+B(u)+bA^-(u)\big( A^+(u)+A^-(u)  \big)\\
& =\lambda\int_{\mathbb{R}^N}f(u^-)u^-dx +\int_{\mathbb{R}^N}|u^-|^{2^{**}}dx.
\end{aligned}
\end{equation}
By Claim 1, we know  that there exists at least  one  positive
pair $(\alpha_0,\beta_0)$ satisfying $\alpha_0u^++\beta_0u^-\in
\mathcal{N}_b^\lambda$.

Next we show  that
$(\alpha_0,\beta_0)=(1,1)$ is the unique pair of numbers. Without
loss of generality, let us assume that $\alpha_0\leq\beta_0$. It follows from
\eqref{e2.8} that
\begin{equation} \label{e2.11}
\begin{aligned}
& \alpha_0^2\left(\|u^+\|^2+B(x)\right)
+\alpha_0^4b A^+(u)\left( A^+(u)+A^-(u)  \right)\\
& =\lambda\int_{\mathbb{R}^N} f(\alpha_0 u^+)\alpha_0 u^+\,dx
+\int_{\mathbb{R}^N}|\alpha_0 u^+|^{2^{**}}dx.
\end{aligned}
\end{equation}
If $\alpha_0<1$, then from \eqref{e2.9}, \eqref{e2.11} and
(A5), we have
\begin{equation} \label{e2.12}
\begin{aligned}
0& <[(\alpha_0)^{-2}-1]\left(\|u^+\|^2+B(u)\right)\\
&\leq \lambda\int_{\mathbb{R}^N}\Big(\frac{f(x,\alpha_0 u^+)}{(\alpha_0
u^+)^3}-\frac{f(u^+)}{(u^+)^3}\Big)(u^+)^4\,dx \\
&\quad +[(\alpha_0)^{2^{**}-4}-1]\int_{\mathbb{R}^N}|u^+|^{2^{**}}dx<0,
\end{aligned}
\end{equation}
which is a contradiction. Hence, $1\leq\alpha_0\leq\beta_0$.

Adopting a similar approach, we can see that $\beta_0\leq 1$, which
implies that $\alpha_0=\beta_0=1$.
\smallskip

\noindent $\bullet$ Case $u \notin \mathcal{N}_b^\lambda$.
Assume there exist two other pairs of positive numbers $(\alpha_1,\beta_1)$ and
$(\alpha_2,\beta_2)$ such that
\[
\sigma_1=\alpha_1u^+ +\beta_1u^-\in\mathcal{N}_b^\lambda\quad \text{and}
\quad \sigma_2=\alpha_2u^++\beta_2u^-\in\mathcal{N}_b^\lambda.
\]
Then
\[
\sigma_2=\big(\frac{\alpha_2}{\alpha_1}\big)\alpha_1u^+
+\big(\frac{\beta_2}{\beta_1}\big)\beta_1u^-
=\big(\frac{\alpha_2}{\alpha_1}\big)\sigma_1^+
+\big(\frac{\beta_2}{\beta_1}\big)\sigma_1^-
\in\mathcal{N}_b^\lambda.
\]
Since $\sigma_1\in\mathcal{N}_b^\lambda$, it is clear that
\[
\frac{\alpha_2}{\alpha_1}=\frac{\beta_2}{\beta_1}=1,
\]
which means that $\alpha_1=\alpha_2$, $\beta_1=\beta_2$.
\smallskip

\noindent\textbf{Claim 3.}
The pair $(\alpha_u,\beta_u)$ is the unique maximum
point of the function $\varphi$ on $\mathbb{R}_+\times\mathbb{R}_+$.
We know from the above that $(\alpha_u,\beta_u)$ is the unique critical point of
 $\varphi$ on $\mathbb{R}_+\times\mathbb{R}_+$. By definition and \eqref{e2.5}, we have
\begin{align*}
\varphi(\alpha,\beta)
&=I_b^\lambda(\alpha u^++\beta u^-)\\
&=\frac{\alpha^2}{2}\|u^
+\|^2+\frac{\beta^2}{2}\|u^-\|^2
+\alpha\beta B(u)
+\frac{\alpha^4 b}{4}\left( A^+(u)\right)^2
+\frac{\beta^4 b}{4}\left( A^-(u)\right)^2\\
&\quad +\frac{\alpha^2\beta^2 b}{2}A^+(u)A^-(u)
-\lambda\int_{\mathbb{R}^N}F(\alpha u^+)\,dx
-\lambda\int_{\mathbb{R}^N}F(\beta u^-)\,dx\\
&\quad -\frac{\alpha^{2^{**}}}{2^{**}}\int_{\mathbb{R}^N}|u^+|^{2^{**}}dx
-\frac{\beta^{2^{**}}}{2^{**}}\int_{\mathbb{R}^N}|u^-|^{2^{**}}dx\\
&<\frac{\alpha^2}{2}\|u^ +\|^2+\frac{\beta^2}{2}\|u^-\|^2
+\alpha\beta B(u)+\frac{\alpha^4 b}{4}\left( A^+(u)\right)^2
+\frac{\beta^4 b}{4}\left( A^-(u)\right)^2\\
&\quad +\frac{\alpha^2\beta^2 b}{2}A^+(u)A^-(u)
-\frac{\alpha^{2^{**}}}{2^{**}}\int_{\mathbb{R}^N}|u^+|^{2^{**}}dx
-\frac{\beta^{2^{**}}}{2^{**}}\int_{\mathbb{R}^N}|u^-|^{2^{**}}dx,
\end{align*}
as $|(\alpha,\beta)|\to\infty$. This implies that
 $\lim_{|(\alpha,\beta)|\to\infty}\varphi(\alpha,\beta)=-\infty$, because
$2^{**}>4$. Hence, it suffices  to show that the maximum point cannot be achieved on
the boundary of $\mathbb{R}_+\times\mathbb{R}_+$.

We carry out the proof by contradiction. Assuming $(0,\bar\beta)$ is the global maximum
point of $\varphi$ with $\bar\beta\ge0$,  we have
\begin{align*}
\varphi(\alpha,\bar\beta)
&=\frac{\alpha^2}{2}\|u^+\|^2+\frac{\bar\beta^2}{2}\|u^-\|^2
+\alpha\bar\beta B(u)
+\frac{\alpha^4 b}{4}\left( A^+(u)\right)^2
+\frac{\bar\beta^4 b}{4}\left( A^-(u)\right)^2\\
&\quad +\frac{\alpha^2\bar\beta^2 b}{2}A^+(u)A^-(u)
-\lambda\int_{\mathbb{R}^N}F(\alpha u^+)\,dx
-\lambda\int_{\mathbb{R}^N}F(\bar\beta u^-)\,dx\\
&\quad -\frac{\alpha^{2^{**}}}{2^{**}}\int_{\mathbb{R}^N}|u^+|^{2^{**}}dx
-\frac{\bar\beta^{2^{**}}}{2^{**}}\int_{\mathbb{R}^N}|u^-|^{2^{**}}dx.
\end{align*}
Hence, it is clear that
\begin{align*}
\varphi'_\alpha(\alpha,\bar\beta)
&=\alpha \|u^+\|^2+\bar\beta B(u)
+\alpha^3 b\left( A^+(u)\right)^2
+\alpha\bar\beta^2bA^+(u)A^-(u)\\
&\quad -\lambda\int_{\mathbb{R}^N}f(\alpha u^+) u^+dx
-\alpha^{2^{**}-1}\int_{\mathbb{R}^N}|u^+|^{2^{**}}dx
>0,
\end{align*}
for small enough $\alpha$. This means
that  $\varphi$ is an increasing function with
 respect to $\alpha$ if $\alpha$ is small enough, which is a contradiction.
In a similar way, we can deduce that $\varphi$ cannot achieve its global maximum at
$(\alpha,0)$ with $\alpha\ge0$. Thus, we have completed the proof.
\end{proof}

\begin{lemma}\label{lem2.2}
Assume that {\rm (A1) ---(A5)} hold. Then for any $u \in E$ with
$u^{\pm} \neq 0$ such that
$\langle(I_b^\lambda)'(u),u^\pm\rangle\leq0$, the unique maximum
point pair of $\varphi$ on $\mathbb{R}_+\times\mathbb{R}_+$ satisfies $0<\alpha_u,
\beta_u\leq 1$.
\end{lemma}

\begin{proof}
Without loss of generality, we may assume that $\alpha_u\ge\beta_u>0$.
Since $\alpha_uu^++\beta_uu^-\in\mathcal{N}_b^\lambda$, we have
\begin{equation} \label{e2.13}
\begin{aligned}
& \alpha_u^2\|u^+\|^2+\alpha_u\beta_u B(u)
+\alpha_u^2\beta_u^2 bA^+(u)A^-(u)
+\alpha_u^4 b\left( A^+(u)\right)^2\\
&= \lambda\int_{\mathbb{R}^N}f(\alpha_u u^+) \alpha_u u^+dx
+\int_{\mathbb{R}^N}|\alpha_u u^+|^{2^{**}}dx.
\end{aligned}
\end{equation}
Furthermore, since $\langle(I_b^\lambda)'(u),u^+\rangle\leq0$, we have
\[
\|u^+\|^2+B(u)+b\left( A^+(u)\right)^2+bA^+(u)A^-(u)
\leq\lambda\int_{\mathbb{R}^N}f(u^+)u^+dx +\int_{\mathbb{R}^N}|u^+|^{2^{**}}dx.
\]
Then by  \eqref{e2.13}, we have
\begin{equation} \label{e2.15}
\begin{aligned}
&[(\alpha_u)^{-2}-1]\left(\|u^+\|^2+B(u)\right)\\
&\ge \lambda\int_{\mathbb{R}^N}\left(\frac{f(\alpha_u u^+)}{(\alpha_u
u^+)^3}-\frac{f(u^+)}{(u^+)^3}\right)(u^+)^4\,dx
+[(\alpha_u)^{2^{**}-4}-1]\int_{\mathbb{R}^N}|u^+|^{2^{**}}dx.
\end{aligned}
\end{equation}
Obviously, the left hand side of \eqref{e2.15} is negative for
$\alpha_u>1$ whereas the right hand side is positive, which is a
contradiction. Therefore $0<\alpha_u,\beta_u\leq 1$.
\end{proof}

\begin{lemma}\label{lem2.3}
Suppose that $c_b^\lambda=\inf_{u\in\mathcal{N}_b^\lambda}I_b^\lambda(u)$.
Then $\lim_{\lambda\to\infty}c_b^\lambda=0$.
\end{lemma}

\begin{proof}
For every $u\in\mathcal{N}_b^\lambda$, we have $\langle(I_b^\lambda)'(u),u\rangle=0$, thus
\[
\|u^+\|^2+\|u^-\|^2+2B(u)+b\left(A^+(u)+A^-(u)\right)^2
=\lambda\int_{\mathbb{R}^N}f(u)u\,dx +\int_{\mathbb{R}^N}|u|^{2^{**}}dx.
\]
Then, by \eqref{e2.5},  we have
\begin{equation} \label{e2.17}
\begin{aligned}
\|u\|^2 & \leq \lambda\int_{\mathbb{R}^N}f(u^\pm)u^\pm\,dx
+\int_{\mathbb{R}^N}|u^\pm|^{2^{**}}dx\\
& \leq \lambda\varepsilon\int_{\mathbb{R}^N}|u^\pm|^2dx
+\lambda C_\varepsilon\int_{\mathbb{R}^N}|u^\pm|^p\,dx
+\int_{\mathbb{R}^N}|u^\pm|^{2^{**}}dx.
\end{aligned}
\end{equation}
Choose $\varepsilon$  small so that
$\lambda\varepsilon\int_{\mathbb{R}^N}|u^\pm|^2dx\leq\frac{1}{2}\|u^\pm\|^2$.
Then we can claim that there exists $\rho>0$ such that
\begin{equation}\label{e2.18}
\|u^\pm\|^2\ge\rho\quad \text{for all } u\in\mathcal{N}_b^\lambda,
\end{equation}
since $4<2^{\ast\ast}$. Next,
by (A5), we have for $t\neq 0$ that
 $$
 \mathcal {F}(t) := tf(t)-4F(t)\ge  0,
 $$
and $\mathcal {F}(t)$ is increasing when $t > 0$, and decreasing when
$t < 0$.
Therefore,
\begin{equation} \label{e2.19}
\begin{aligned}
I_b^\lambda(u)
& = I_b^\lambda(u)-\frac{1}{4}\langle(I_b^\lambda)'(u),u\rangle \\
& = \frac{1}{4}\|u\|^2
+(\frac{1}{4}-\frac{1}{2^{**}})\int_{\mathbb{R}^N}|u|^{2^{**}}dx
+\frac{\lambda}{4}\int_{\mathbb{R}^N}\left[ f(u)u-4F(u)\right]dx \\
& \ge\frac{1}{4}\|u\|^2
\ge\frac{\rho}{4}>0.
\end{aligned}
\end{equation}
So we have $I_b^\lambda(u)>0$ for all $u\in\mathcal{N}_b^\lambda$, which means
that  $c_b^\lambda=\inf_{u\in\mathcal{N}_b^\lambda}I_b^\lambda(u)$ is well-defined.

Fix $u\in E$ with $u^\pm\neq0$. According to Lemma \ref{lem2.1}, for
each $\lambda>0$, there exist $\alpha_\lambda, \beta_\lambda>0$ such
that $\alpha_\lambda u^++\beta_\lambda u^-\in\mathcal{N}_b^\lambda$.
Therefore,
\begin{align*}
0\leq c_b^\lambda
&= \inf_{u\in\mathcal{N}_b^\lambda} I_b^\lambda(u)\\
&\leq I_b^\lambda(\alpha_\lambda u^++\beta_\lambda u^-)\\
&\leq \frac{1}{2}\|\alpha_\lambda u^++\beta_\lambda u^-\|^2
+\frac{b}{4}\Big(\int_{\mathbb{R}^N}|\nabla(\alpha_\lambda u^++\beta_\lambda u^-)|^2dx\Big)^2 \\
&= \frac{\alpha_\lambda^2}{2}\|u^+\|^2
+\frac{\beta_\lambda^2}{2}\|u^-\|^2
+\alpha_\lambda\beta_\lambda B(u)
+\frac{\alpha_\lambda^4 b}{4}\left(A^+(u)\right)^2\\
&\quad +\frac{\beta_\lambda^4 b}{4}\left(A^- (u)\right)^2
+\frac{\alpha_\lambda^2\beta_\lambda^2 b}{2}A^+(u)A^-(u).
\end{align*}

It suffices to prove that $\alpha_\lambda\to0$ and $\beta_\lambda\to0$, as $\lambda\to\infty$.
Let
$$
\mathcal{T}=\{(\alpha_\lambda, \beta_\lambda)\in\mathbb{R}_+\times\mathbb{R}_+:
W(\alpha_\lambda, \beta_\lambda)=(0,0), \lambda>0\},
$$
where $W$ is defined as in \eqref{e2.2}. Then
\begin{align*}
&\alpha_\lambda^{2^{**}}\int_{\mathbb{R}^N}|u^+|^{2^{**}}dx
+\beta_\lambda^{2^{**}}\int_{\mathbb{R}^N}|u^-|^{2^{**}}dx
\\
&\leq
\alpha_\lambda^{2^{**}}\int_{\mathbb{R}^N}|u^+|^{2^{**}}dx
+\beta_\lambda^{2^{**}}\int_{\mathbb{R}^N}|u^-|^{2^{**}}dx
\\
&\quad +\lambda\int_{\mathbb{R}^N}f(\alpha_\lambda u^+)\alpha_\lambda u^+dx
+\lambda\int_{\mathbb{R}^N}f(\beta_\lambda u^-)\beta_\lambda u^-dx
\\
&=
\|\alpha_\lambda u^++\beta_\lambda u^-\|^2
+b\left(\alpha_\lambda^2 A^+(u)+\beta_\lambda^2 A^-(u)\right)^2.
\end{align*}
Therefore, $\mathcal{T}$ is bounded, since $4<2^{**}$. Let
$\{\lambda_n\}\subset(0, \infty)$ be such that $\lambda_n\to\infty$,
as $n\to\infty$. Then  there exist $\alpha_0$ and $\beta_0$
such that $(\alpha_{\lambda_n},\beta_{\lambda_n})\to(\alpha_0,
\beta_0)$, as $n\to\infty$.

Now we claim that $\alpha_0=\beta_0=0$. Assume, to the contrary, that
$\alpha_0>0$ or $\beta_0>0$. Since
$\alpha_{\lambda_n}u^++\beta_{\lambda_n}u^-\in\mathcal{N}_b^{\lambda_n}$,
then for any $n\in\mathbb{N}$, we have
\begin{equation} \label{e2.20}
\begin{aligned}
&\|\alpha_{\lambda_n}u^++\beta_{\lambda_n}u^-\|^2
+b\left(\alpha_{\lambda_n}^2 A^+(u)+\beta_{\lambda_n}^2 A^-(u)\right)^2\\
&=\lambda_n\int_{\mathbb{R}^N}f(\alpha_{\lambda_n}u^++\beta_{\lambda_n}u^-)(\alpha_{\lambda_n}u^++\beta_{\lambda_n}u^-)\,dx \\
&\quad +\int_{\mathbb{R}^N}|\alpha_{\lambda_n}u^++\beta_{\lambda_n}u^-|^{2^{**}}dx.
\end{aligned}
\end{equation}
Then, invoking $\alpha_{\lambda_n}u^+\to\alpha_0
u^+, \beta_{\lambda_n}u^-\to\beta_0 u^-$ in $E$ and the Lebesgue
dominated convergence theorem, we have
\begin{align*}
&\int_{\mathbb{R}^N}f(\alpha_{\lambda_n}u^++\beta_{\lambda_n}u^-)(\alpha_{\lambda_n}u^++\beta_{\lambda_n}u^-)\,dx
\\
&\to \int_{\mathbb{R}^N}f(\alpha_0 u^++\beta_0 u^-)(\alpha_0 u^++\beta_0
u^-)\,dx>0,
\end{align*}
as $n\to\infty$. This  contradicts \eqref{e2.20}, given that
$\lambda_n\to\infty$, as $n\to\infty$ and that
$\{\alpha_{\lambda_n}u^++\beta_{\lambda_n}u^-\}$ is bounded in $E$.
Therefore, $\alpha_0=\beta_0=0$, which implies
$\lim_{\lambda\to\infty}c_b^\lambda=0$.
\end{proof}

\begin{lemma}\label{lem2.4}
There exists $\lambda^*>0$ such that the infimum $c_b^\lambda$ is achieved for
all $\lambda\ge\lambda^*$.
\end{lemma}

\begin{proof}
According to the definition of $c_b^\lambda$, there exists a
sequence $\{u_n\}\subset\mathcal{N}_b^\lambda$ such that
$\lim_{n\to\infty}I_b^\lambda(u_n)=c_b^\lambda$. Clearly, $\{u_n\}$
is bounded in $E$. By Lemma \ref{lem2.1} and the properties of $L^p$
space, up to a subsequence, we have
\begin{gather*}
u_n^\pm\rightharpoonup u^\pm\quad \text{in } E,\\
u_n^\pm\to u^\pm\quad \text{in $ L^p(\mathbb{R}^N)$ for } p\in[2,2^{**}),\\
u_n^\pm\to u^\pm\quad \text{a.e.\ in } \mathbb{R}^N.
\end{gather*}
In view of Lemma \ref{lem2.1}, we also have
\[
I_b^\lambda(\alpha u_n^++\beta u_n^-) \leq I_b^\lambda(u_n),
\]
for all $\alpha, \beta\ge0$. So, by the Br\'{e}zis-Lieb lemma, Fatou's
lemma and the weak lower semicontinuity of norm, we can conclude that
\begin{align*}
&\liminf_{n\to\infty} I_b^\lambda(\alpha u_n^++\beta u_n^-)\\
&\ge \frac{\alpha^2}{2}\lim_{n\to\infty} (\|u_n^+-u^+\|^2+\|u^+\|^2)
+\frac{\beta^2}{2}\lim_{n\to\infty} (\|u_n^--u^-\|^2+\|u^-\|^2)\\
&\quad +\frac{\alpha^4b}{4}\Big[\lim_{n\to\infty}
\int_{\mathbb{R}^N}|\nabla u_n^+-\nabla u^+|^2\,dx
+\int_{\mathbb{R}^N}|\nabla u^+|^2 \,dx \Big]^2
\\
&\quad +\frac{\beta^4b}{4}\Big[\lim_{n\to\infty}
\int_{\mathbb{R}^N}|\nabla u_n^--\nabla u^-|^2\,dx
+\int_{\mathbb{R}^N}|\nabla u^-|^2 \,dx \Big]^2
\\
&\quad -\frac{\alpha^{2^{**}}}{2^{**}}
\Big[\lim_underset{n\to\infty} \int_{\mathbb{R}^N}|u_n^+-u^+|^{2^{**}}dx
+\lim_{n\to\infty} \int_{\mathbb{R}^N}|u^+|^{2^{**}}dx\Big]
\\
&\quad -\frac{\beta^{2^{**}}}{2^{**}}
 \Big[\lim_{n\to\infty} \int_{\mathbb{R}^N}|u_n^--u^-|^{2^{**}}dx
+\lim_{n\to\infty} \int_{\mathbb{R}^N}|u^-|^{2^{**}}dx\Big]
\\
&\quad -\lambda\int_{\mathbb{R}^N}F(\alpha u^+)\,dx
-\lambda\int_{\mathbb{R}^N}F(\beta u^-)\,dx
\\
&\quad +\frac{\alpha^2\beta^2 b}{2}\liminf_{n\to\infty}
\int_{\mathbb{R}^N}|\nabla u_n^+|^2dx\int_{\mathbb{R}^N}|\nabla u_n^-|^2dx
\\
&\ge I_b^\lambda(\alpha u^++\beta u^-)
+\frac{\alpha^2}{2}A_1
+\frac{\alpha^4b}{4}A_3^2
+\frac{\alpha^4b}{2}A_3A^+(u)
-\frac{\alpha^{2^{**}}}{2^{**}}B_1
\\
&\quad +\frac{\beta^2}{2}A_2
+\frac{\beta^4b}{4}A_4^2
+\frac{\beta^4b}{2}A_4A^-(u)
-\frac{\beta^{2^{**}}}{2^{**}}B_2,
\end{align*}
where
\begin{gather*}
A_1=\lim_{n\to\infty} \|u_n^+-u^+\|^2, \quad
A_2=\lim_{n\to\infty} \|u_n^--u^-\|^2,
\\
A_3=\lim_{n\to\infty} \int_{\mathbb{R}^N}|\nabla u_n^+-\nabla u^+|^2dx,\quad
A_4=\lim_{n\to\infty} \int_{\mathbb{R}^N}|\nabla u_n^--\nabla u^-|^2dx,
\\
B_1=\lim_{n\to\infty} \int_{\mathbb{R}^N}|u_n^+-u^+|^{2^{**}}dx,\quad
B_2=\lim_{n\to\infty} \int_{\mathbb{R}^N}|u_n^--u^-|^{2^{**}}dx.
\end{gather*}
That is,
\begin{equation} \label{e2.22}
\begin{aligned}
c_b^\lambda
&\ge I_b^\lambda(\alpha u^++\beta u^-)
+\frac{\alpha^2}{2}A_1
+\frac{\alpha^4b}{4}A_3^2
+\frac{\alpha^4b}{2}A_3A^+(u)
-\frac{\alpha^{2^{**}}}{2^{**}}B_1
\\
&\quad +\frac{\beta^2}{2}A_2 +\frac{\beta^4b}{4}A_4^2
+\frac{\beta^4b}{2}A_4A^-(u) -\frac{\beta^{2^{**}}}{2^{**}}B_2,
\end{aligned}
\end{equation}
for all $\alpha, \beta\ge0$.
\smallskip

\noindent\textbf{Step 1:  $u^\pm \neq 0$.}
We carry out our proof by contradiction. Assume that $u^+=0$. Lettong  $\beta=0$ in \eqref{e2.22}  we have
\begin{equation} \label{e2.23}
c_b^\lambda  \ge \frac{\alpha^2}{2}A_1 +\frac{\alpha^4b}{4}A_3^2
-\frac{\alpha^{2^{**}}}{2^{**}}B_1 := \phi(\alpha),
\end{equation}
for all $\alpha\ge0$.
\smallskip

\noindent\textbf{Case 1: $B_1=0$.}
If $A_1=0$, then $u_n^+\to u^+$ in $E$. By \eqref{e2.17}, we
obtain $\|u^\pm\|>0$, which  contradicts our assumption.
If $A_1>0$, then by \eqref{e2.23}, we have
$c_b^\lambda\ge\frac{\alpha^2}{2}A_1$ for all $\alpha\ge0$, which
contradicts Lemma \ref{lem2.3}.
\smallskip

\noindent\textbf{Case 2: $B_1>0$.}
From the definition of $S$ and Lemma \ref{lem2.3}, there exists $\lambda^*>0$ such that
\begin{equation}\label{e2.24}
c_b^\lambda<\frac{2}{N}S^{-2/N}
\end{equation}
for all $\lambda\ge\lambda^*$. According to the Sobolev embedding and the fact that $B_1>0$, we obtain $A_1>0$. By \eqref{e2.23}, we have
\begin{align*}
\frac{2}{N}S^{-2/N}
&\leq \frac{2}{N}\Big[\frac{A_1^{\frac{2^{**}}{2}}}{B_1}\Big]^{\frac{2}{2^{**}-2}}\\
&\leq\max_{\alpha\ge0} \big\{\frac{\alpha^2}{2}A_1-\frac{\alpha^{2^{**}}}{2^{**}}B_1\big\}
\\
&\leq \max_{\alpha\ge0} \big\{\frac{\alpha^2}{2}A_1+\frac{\alpha^4b}{4}A_3^2-\frac{\alpha^{2^{**}}}{2^{**}}B_1\big\}
\leq c_b^\lambda,
\end{align*}
which is a contradiction. Hence, we can conclude that $u^+\neq0$. Similarly, we get that $u^-\neq0$.
\smallskip

\noindent\textbf{Step 2:  $B_1=B_2=0$.}
Given that the proof of $B_2=0$ is analogous, we just prove $B_1=0$.
By contradiction, assume $B_1>0$.
\smallskip

\noindent\textbf{Case 1: $B_2>0$.}
Since $B_1, B_2>0$, we get $A_1, A_2>0$. Clearly, $\phi(\alpha)>0$ for
$\alpha$ small enough, where $\phi(\alpha)$ is given by \eqref{e2.23}, and
$\phi(\alpha)<0$ for $\alpha$ sufficiently large.
Therefore, by continuity of $\phi(\alpha)$, there exists $\bar\alpha>0$ such that
\[
\frac{\bar\alpha^2}{2}A_1
+\frac{\bar\alpha^4b}{4}A_3^2
-\frac{\bar\alpha^{2^{**}}}{2^{**}}B_1
=
\max_{\alpha\ge0} \big\{\frac{\alpha^2}{2}A_1+\frac{\alpha^4b}{4}A_3^2-\frac{\alpha^{2^{**}}}{2^{**}}B_1\big\}.
\]
Similarly, there exists $\bar\beta>0$ such that
\[
\frac{\bar\beta^2}{2}A_2
+\frac{\bar\beta^4b}{4}A_4^2
-\frac{\bar\beta^{2^{**}}}{2^{**}}B_2
=
\max_{\beta\ge0} \big\{\frac{\beta^2}{2}A_2
+\frac{\beta^4b}{4}A_4^2-\frac{\beta^{2^{**}}}{2^{**}}B_2\big\}.
\]
In view of the compactness of $[0,\bar\alpha]\times[0,\bar\beta]$
and the continuity of $\phi$, there exists
$(\alpha_u,\beta_u)\in[0,\bar\alpha]\times[0,\bar\beta]$ such that
$$
\varphi(\alpha_u,\beta_u)=\underset{(\alpha,\beta)\in[0,\bar\alpha]\times[0,\bar\beta]}{\max}\varphi(\alpha,\beta),
$$
where $\varphi$ is defined as in Lemma \ref{lem2.1}.

Now we prove that $(\alpha_u,\beta_u)\in(0,\bar\alpha)\times(0,\bar\beta)$.
Note that if $\beta$ is small enough, then we have
\[
\varphi(\alpha,0)
=I_b^\lambda(\alpha u^+)
<I_b^\lambda(\alpha u^+)+I_b^\lambda(\beta u^-)
\leq I_b^\lambda(\alpha u^++\beta u^-)
=\varphi(\alpha,\beta),
\]
for all $\alpha\in[0,\bar\alpha]$. Thus, there exists
$\beta_0\in[0,\bar\beta]$ such that
$\varphi(\alpha,0)\leq\varphi(\alpha,\beta_0)$, for all
$\alpha\in[0,\bar\alpha]$. That is,
$(\alpha_u,\beta_u)\notin[0,\bar\alpha]\times\{0\}$. With a similar
method, we can show that $(\alpha_u,\beta_u)\notin\{0\}\times[0,\bar\beta]$.

It is obvious that
\begin{align}\label{e2.25}
\frac{\alpha^2}{2}A_1
+\frac{\alpha^4b}{4}A_3^2
+\frac{\alpha^4b}{2}A_3A^+(u)
-\frac{\alpha^{2^{**}}}{2^{**}}B_1
>0,\quad \alpha\in(0,\bar\alpha]
\end{align}
and
\begin{align}\label{e2.26}
\frac{\beta^2}{2}A_2
+\frac{\beta^4b}{4}A_4^2
+\frac{\beta^4b}{2}A_4A^-(u)
-\frac{\beta^{2^{**}}}{2^{**}}B_2
>0,\quad \beta\in(0,\bar\beta].
\end{align}
Thus we obtain
\begin{align*}
\frac{2}{N}S^{-2/N}
&\leq \frac{\bar\alpha^2}{2}A_1
+\frac{\bar\alpha^4b}{4}A_3^2 -\frac{\bar\alpha^{2^{**}}}{2^{**}}B_1
+\frac{\bar\alpha^4b}{2}A_3A^+(u)
\\
&\quad
+\frac{\beta^2}{2}A_2
+\frac{\beta^4b}{4}A_4^2 +\frac{\beta^4b}{2}A_4A^-(u)
-\frac{\beta^{2^{**}}}{2^{**}}B_2
\end{align*}
and
\begin{align*}
\frac{2}{N}S^{-2/N}
&\leq \frac{\bar\beta^2}{2}A_2
+\frac{\bar\beta^4b}{4}A_4^2 -\frac{\bar\beta^{2^{**}}}{2^{**}}B_2
+\frac{\bar\beta^4b}{2}A_4A^-(u)
\\
&\quad
+\frac{\alpha^2}{2}A_1
+\frac{\alpha^4b}{4}A_3^2 +\frac{\alpha^4b}{2}A_3A^+(u)
-\frac{\alpha^{2^{**}}}{2^{**}}B_1,
\end{align*}
for all $\alpha\in[0,\bar\alpha]$, $\beta\in[0,\bar\beta]$.

From the these inequalities and  \eqref{e2.22}, we obtain
$\varphi(\bar\alpha,\beta)\leq0$, $\varphi(\alpha,\bar\beta)\leq0$ for all
$\alpha\in[0,\bar\alpha]$, $\beta\in[0,\bar\beta]$.
Therefore, $(\alpha_u,\beta_u)\notin \{\bar\alpha\}\times[0,\bar\beta]$ and
$(\alpha_u,\beta_u)\notin [0,\bar\alpha]\times\{\bar\beta\}$, which means
$(\alpha_u,\beta_u)\in (0,\bar\alpha)\times(0,\bar\beta)$. It follows that
$(\alpha_u,\beta_u)$ is a critical point of $\varphi$.

So, $\alpha_uu^++\beta_uu^-\in\mathcal{N}_b^\lambda$.
By \eqref{e2.22}, we have
\begin{align*}
c_b^\lambda
&\ge I_b^\lambda(\alpha_u u^++\beta_u u^-)
+\frac{\alpha_u^2}{2}A_1
+\frac{\alpha_u^4b}{4}A_3^2
+\frac{\alpha_u^4b}{2}A_3A^+(u)
-\frac{\alpha_u^{2^{**}}}{2^{**}}B_1
\\
&\quad +\frac{\beta_u^2}{2}A_2
+\frac{\beta_u^4b}{4}A_4^2
+\frac{\beta_u^4b}{2}A_4A^-(u)
-\frac{\beta_u^{2^{**}}}{2^{**}}B_2
>I_b^\lambda(\alpha_u u^++\beta_u u^-)
\ge c_b^\lambda,
\end{align*}
which is a contradiction. Therefore $B_1=0$.
\smallskip

\noindent\textbf{Case 2: $B_2 = 0$.}
In this case, we can maximize in $[0,\bar\alpha]\times[0,\infty)$.
It is possible to show that there exists $\beta_0\in[0,\infty)$ satisfying
\[
I_b^\lambda(\alpha_u u^++\beta_u u^-)\leq0\quad \text{for all} \,\,(\alpha,\beta)\in[0,\bar\alpha]\times[\beta_0,\infty).
\]
Then  there is $(\alpha_u,\beta_u)\in[0,\bar\alpha]\times[0,\infty)$ such that
\[
\varphi(\alpha_u,\beta_u)=\underset{(\alpha,\beta)\in[0,\bar\alpha]\times[0,\infty)}{\max}\varphi(\alpha,\beta).
\]

We claim that $(\alpha_u,\beta_u)\in(0,\bar\alpha)\times(0,\infty)$.
Indeed, $\varphi(\alpha,0)<\varphi(\alpha,\beta)$ for $\alpha\in[0,\bar\alpha]$ and
$\beta$ small enough, while $\varphi(0,\beta)<\varphi(\alpha,\beta)$ for
$\beta\in[0,\infty)$ and $\alpha$ sufficiently small, which implies
$(\alpha_u,\beta_u)\notin[0,\bar\alpha]\times\{0\}$ and
$(\alpha_u,\beta_u)\notin\{0\}\times[0,\infty)$.

Note  that
\[
\frac{2}{N}S^{-2/N}
\leq
\frac{\bar\alpha^2}{2}A_1
+\frac{\bar\alpha^4b}{4}A_3^2
-\frac{\bar\alpha^{2^{**}}}{2^{**}}B_1
+\frac{\bar\alpha^4b}{2}A_3A^+(u)
+\frac{\beta^2}{2}A_2
+\frac{\beta^4b}{4}A_4^2
+\frac{\beta^4b}{2}A_4A^-(u),
\]
for every $\beta\in[0,\infty)$.
Therefore, we have $\varphi(\bar\alpha,\beta)\leq0$ for all $\beta\in[0,\infty)$,
which means $(\alpha_u,\beta_u)\notin\{\bar\alpha\}\times[0,\infty)$.
Based on the above, we get $(\alpha_u,\beta_u)\in(0,\bar\alpha)\times(0,\infty)$,
that is, $(\alpha_u,\beta_u)$ is an inner maximizer of $\varphi$ in $[0,\bar\alpha]\times[0,\infty)$.
Therefore, $\alpha_u u^++\beta_u u^-\in\mathcal{N}_b^\lambda$. In that case, by \eqref{e2.25}, we have
\begin{align*}
c_b^\lambda
&\ge
I_b^\lambda(\alpha_u u^++\beta_u u^-)
+\frac{\bar\alpha^2}{2}A_1
+\frac{\bar\alpha^4b}{4}A_3^2
-\frac{\bar\alpha^{2^{**}}}{2^{**}}B_1
\\
&\quad +\frac{\bar\alpha^4b}{2}A_3A^+(u)
+\frac{\beta^2}{2}A_2
+\frac{\beta^4b}{4}A_4^2
+\frac{\beta^4b}{2}A_4A^-(u)
\\
&>I_b^\lambda(\alpha_u u^++\beta_u u^-)
\ge c_b^\lambda,
\end{align*}
which  is a contradiction. Hence, we have $B_1=B_2=0$.
\smallskip

\noindent\textbf{Step 3:  $c_b^\lambda$ is achieved.}
Given $u^\pm\neq0$, according to Lemma \ref{lem2.1}, there exists
$\alpha_u,\beta_u>0$ such that $\hat{u}:=\alpha_u u^++\beta_u u^-\in\mathcal{N}_b^\lambda$.
Moreover, $\langle(I_b^\lambda)'(u),u^\pm\rangle\leq0$. By Lemma \ref{lem2.2},
we have $0<\alpha_u, \beta_u\leq1$.

Combining $u_n\in\mathcal{N}_b^\lambda$ and Lemma \ref{lem2.1}, we obtain
\[
I_b^\lambda(\alpha_u u_n^++\beta_u u_n^-)
\leq I_b^\lambda(u_n^++u_n^-)
=I_b^\lambda(u_n).
\]

Taking into consideration $B_1=B_2=0$ and the semicontinuity of the
norm, we obtain
\begin{align*}
c_b^\lambda
&\leq I_b^\lambda(\hat{u}) \\
&= I_b^\lambda(\hat{u})
-\frac{1}{4}\langle(I_b^\lambda)'(\hat{u}),\hat{u}\rangle
\\
& = \frac{1}{4}\|\hat{u}\|^2
+(\frac{1}{4}-\frac{1}{2^{**}})\int_{\mathbb{R}^N}|\hat{u}|^{2^{**}}dx
+\frac{\lambda}{4}\int_{\mathbb{R}^N}\left[f(\hat{u})\hat{u}-4F(\hat{u})\right]dx
\\
& \leq \frac{1}{4}\|u\|^2
+(\frac{1}{4}-\frac{1}{2^{**}})\int_{\mathbb{R}^N}|u|^{2^{**}}dx
+\frac{\lambda}{4}\int_{\mathbb{R}^N}\left[f(u)u-4F(u)\right]dx
\\
& \leq
\underset{n\to\infty}{\lim\inf}\left[I_b^\lambda(u_n)
-\frac{1}{4}\langle(I_b^\lambda)'(u_n),u_n\rangle\right]
\leq
c_b^\lambda.
\end{align*}
Hence, we can conclude that  $\alpha_u=\beta_u=1$, and $c_b^\lambda$
 is achieved by $u_b:=u^++u^-\in\mathcal{N}_b^\lambda$.
\end{proof}

\section{Proofs of main results}

\begin{proof}[Proof of Theorem \ref{th1.1}]
Thanks to Lemma \ref{lem2.4}, we only need to prove that the
minimizer $u_b$ for $c_b^\lambda$ is indeed a nodal solution to
problem \eqref{e1.1}.

Because $u_b\in\mathcal{N}_b^\lambda$, we have
$\langle(I_b^\lambda)'(u_b),u_b^+\rangle=\langle(I_b^\lambda)'(u_b),u_b^-\rangle=0$.
In view of Lemma \ref{lem2.1}, for
$(\alpha,\beta)\in(\mathbb{R}_+\times\mathbb{R}_+)\backslash (1,1)$, we have
\begin{equation}\label{e3.1}
I_b^\lambda(\alpha u_b^+ + \beta u_b^-)
<I_b^\lambda(u_b^+ +u_b^-)
=c_b^\lambda.
\end{equation}

Now we proceed by contradiction.
Suppose $(I_b^\lambda)'(u_b)\neq0$, then there exist $\delta>0$ and
$\theta>0$ such that
\[
\|(I_b^\lambda)'(v)\|\geq\theta \quad \text{for all}\,\,\|v-u_b\|
\leq 3\delta.
\]
Choose $\tau \in(0,\min\{1/2,\frac{\delta}{\sqrt{2}\|u_b\|}\})$, and define
\begin{gather*}
D:= (1-\tau,1+\tau)\times(1-\tau,1+\tau),
\\
g(\alpha,\beta):= \alpha u_b^+ + \beta u_b^-\quad \text{for all}\, \, (\alpha,\beta)\in D.
\end{gather*}
By \eqref{e3.1}, we have
\begin{equation}\label{e3.2}
\bar{c}_\lambda:=\max_{\partial D}(I_b^\lambda\circ g)<c_b^\lambda.
\end{equation}
Let $\varepsilon:= \min\{(c_b^\lambda
-\bar{c}_\lambda)/2,\theta\delta/8\}$ and $S_\delta:= B(u_b,\delta)$.
By  \cite[Lemma 2.3]{mw1996},
there exists a deformation $\eta\in C([0,1]\times D,D)$  such that
\begin{itemize}
\item [(a)] $\eta(1,v) = v$ if $v\notin(I_b^\lambda)^{-1}([c_b^\lambda-2\varepsilon, c_b^\lambda+2\varepsilon]\cap
S_{2\delta})$,

\item[(b)] $\eta(1,(I_b^\lambda)^{c_b^\lambda+\varepsilon}\cap
S_{\delta}) \subset (I_b^\lambda)^{c_b^\lambda-\varepsilon}$,

\item[(c)] $I_b^\lambda(\eta(1,v))\leq I_b^\lambda(v)$ for all $v\in E$.
\end{itemize}

Clearly,
\begin{equation}\label{e3.3}
\max_{(\alpha,\beta)\in\bar{D}}I_b^\lambda(\eta(1,g(\alpha,\beta)))<
c_b^\lambda.
\end{equation}
Therefore we claim that $\eta(1,g(D))\cap
\mathcal{N}_b^\lambda \neq \emptyset$ , which  contradicts the definition of $c_b^\lambda$.

We define $h(\alpha,\beta):=\eta(1,g(\alpha,\beta))$,
\begin{align*}
\Phi_0(\alpha,\beta)
:&=(\langle(I_b^\lambda)'(g(\alpha,\beta)),u_b^+\rangle,\langle(I_b^\lambda)'(g(\alpha,\beta)), u_b^-\rangle)\\
& = (\langle(I_b^\lambda)'(\alpha u_b^+ + \beta u_b^-),
u_b^+\rangle, \langle(I_b^\lambda)'(\alpha u_b^+ + \beta
u_b^-),u_b^-\rangle)
\end{align*}
and
\[
\Phi_1(\alpha,\beta) :=
\Big(\frac{1}{\alpha}\langle(I_b^\lambda)'(h(\alpha,\beta)),\,
(h(\alpha,\beta))^+\rangle,\frac{1}{\beta}\langle(I_b^\lambda)'(h(\alpha,\beta)),(h(\alpha,\beta))^-\rangle\Big).
\]

With an approach similar to \cite{liang4}, we use degree theory to obtain
$\deg(\Phi_0 ,D,0) = 1$. Then by \eqref{e3.2}, we obtain
\[
g(\alpha,\beta)
= h(\alpha,\beta)\quad \text{on } \partial D,
\]
as a result of which, we have $\deg(\Phi_1 ,D,0) =\deg(\Phi_0 ,D,0) = 1$.
Hence, $\Phi_1(\alpha_0, \beta_0) = 0$ for some $(\alpha_0, \beta_0)\in D$
so that
\[
\eta(1,g(\alpha_0, \beta_0))=h(\alpha_0, \beta_0)\in
\mathcal{N}_b^\lambda,
\]
which contradicts \eqref{e3.3}. Hence, $(I_b^\lambda)'(u_b)=0$,
which implies $u_b$ is a critical point of $I_b^\lambda$. Thus, we
can deduce that  $u_b$ is a nodal solution to problem \eqref{e1.1}.
\end{proof}

By Theorem \ref{th1.1}, we obtain a least energy nodal solution
$u_b$ to problem \eqref{e1.1}, contributing to the establishment of
Theorem \ref{th1.2}, where we shall prove that the energy of $u_b$ is
strictly larger than twice the ground state energy.

\begin{proof}[Proof of Theorem \ref{th1.2}]
As in the proof of Lemma \ref{lem2.3}, there exists $\lambda^*_1 >0$
such that for all $\lambda\geq\lambda^*_1$, and for each $b > 0$,
there exists $v_b\in \mathcal{M}_b^\lambda$ such that
$I_b^\lambda(v_b)=c^*>0$. By standard arguments
(see \cite[Corollary 2.13]{he1}), the critical points of the functional
$I_b^\lambda$ on $\mathcal{M}_b^\lambda$ are critical points of
$I_b^\lambda$ in $E$, so we obtain $(I_b^\lambda)'(v_b)=0$. That is,
$v_b$ is a ground state solution to problem \eqref{e1.1}.

As stated in Theorem \ref{th1.1}, $u_b$ is known as a least energy
nodal solution to problem \eqref{e1.1}, which changes sign only once
when $\lambda\geq\lambda^*$.

Let $\lambda^{**}=\max\{\lambda^*,\lambda^*_1\}$ and assume
$u_b=u_b^++u_b^-$. Adopting the same approach as in  Lemma \ref{lem2.1},
we claim there exist $\alpha_{u_b^+} >0$ and $\beta_{u_b^-}>0$ such
that $\alpha_{u_b^+}u_b^+\in\mathcal{M}_b^\lambda$ and
$\beta_{u_b^-}u_b^-\in \mathcal{M}_b^\lambda$. Then, by Lemma
\ref{lem2.2}, we obtain $\alpha_{u_b^+}, \beta_{u_b^-}\in(0, 1)$.
Hence, thanks to Lemma \ref{lem2.1}, we have
\[
2c^*
\leq I_b^\lambda(\alpha_{u_b^+}u_b^+) + I_b^\lambda(\beta_{u_b^-}u_b^-)
\leq I_b^\lambda(\alpha_{u_b^+}u_b^++\beta_{u_b^-}u_b^-)
<I_b^\lambda(u_b^+ + u_b^-)
= c_b^\lambda.
\]
It follows that $c^*>0$ cannot be achieved by a nodal
function.
\end{proof}

We complete this section with the proof of Theorem
\ref{th1.3}. In the sequel, we regard $b > 0$ as a parameter in
problem \eqref{e1.1}.

\begin{proof}[Proof of Theorem \ref{th1.3}]
 In 3 steps, we analyze the convergence property of $u_b$ as $b
\to 0$, where $u_b$ is the least energy nodal solution obtained in
Theorem \ref{th1.1}.
\smallskip

\noindent\textbf{Step 1.}
For any sequence $\{b_n\}$, we prove that $\{u_{b_n}\}$ is
bounded in $E$, if $b_n \searrow 0$.
Let  $\chi \in C_0^\infty(\mathbb{R}^N)$ be a nonzero function with
$\chi^\pm\neq0$ fixed. Analogous to the argument in Lemma
\ref{lem2.1}, for any $b \in [0, 1]$, there exists a pair of
positive numbers $(\lambda_1, \lambda_2)$ independent of $b$, such
that
\[
\langle(I_b^\lambda)'(\lambda_1\chi^+ + \lambda_2\chi^-),
\lambda_1\chi^+ \rangle < 0 \quad\text{and}\quad
\langle(I_b^\lambda)'(\lambda_1\chi^+ + \lambda_2\chi^-),
\lambda_2\chi^- \rangle < 0.
\]

Then according to Lemma \ref{lem2.2}, for any $b \in[0, 1]$, there
exists a unique pair $(\alpha_\chi(b), \beta_\chi(b)) \in (0, 1]
\times (0, 1]$ such that $\overline{\chi} :=
\alpha_\chi(b)\lambda_1\chi^++\beta_\chi(b)\lambda_2\chi^-\in
\mathcal{N}_b^\lambda$. Therefore, by \eqref{e2.5}, it follows that, for any
$b\in [0, 1]$,
\begin{align*}
I_b^\lambda(u_b) \leq I_b^\lambda(\overline{\chi})
& =
I_b^\lambda(\overline{\chi})
-\frac{1}{4}\langle(I_b^\lambda)'(\overline{\chi}),
\overline{\chi}\rangle
\\
& = \frac{1}{4}\|\overline{\chi}\|^2
+(\frac{1}{4}-\frac{1}{2^{**}})\int_{\mathbb{R}^N}|\overline{\chi}|^{2^{**}}dx
+\frac{\lambda}{4}\int_{\mathbb{R}^N}\left[f(\overline{\chi})\overline{\chi}-4F(\overline{\chi})\right]dx
\\
&\leq \frac{1}{4}\|\overline{\chi}\|^2
+(\frac{1}{4}-\frac{1}{2^{**}})\int_{\mathbb{R}^N}|\overline{\chi}|^{2^{**}}dx
+\frac{\lambda}{4}\int_{\mathbb{R}^N}\left(C_1|\overline{\chi}|^2+C_2|\overline{\chi}|^p\right)\,dx
\\
&\leq \frac{1}{4}\|\lambda_1\chi^+\|^2
+(\frac{1}{4}-\frac{1}{2^{**}})\int_{\mathbb{R}^N}|\lambda_1\chi^+|^{2^{**}}dx
\\
&\quad
+\frac{\lambda}{4}\int_{\mathbb{R}^N}\left(C_1|\lambda_1\chi^+|^2+C_2|\lambda_1\chi^+|^p\right)\,dx
\\
&\quad +\frac{1}{4}\|\lambda_2\chi^-\|^2
+(\frac{1}{4}-\frac{1}{2^{**}})\int_{\mathbb{R}^N}|\lambda_2\chi^-|^{2^{**}}dx
\\
&\quad
+\frac{\lambda}{4}\int_{\mathbb{R}^N}\left(C_1|\lambda_2\chi^-|^2+C_2|\lambda_2\chi^-|^p\right)\,dx
\\
&:= C^*,
\end{align*}
where $C^*$ is a positive constant independent of $b$.
Thus, as $n\to\infty$, it follows that
\[
C^*+ 1
\ge I_{b_n}^\lambda(u_{b_n})
= I_{b_n}^\lambda(u_{b_n})- \frac{1}{4}\langle (I_{b_n}^\lambda)'(u_{b_n}), u_{b_n}\rangle
\ge \frac{1}{4} \|u_{b_n}\|^2,
\]
that is, $\{u_{b_n}\}$ is bounded in $E$.
\smallskip

\noindent\textbf{Step 2.} In this step, we prove  that problem \eqref{e1.14} possesses
one nodal solution $u_0$.
Since $\{u_{b_n}\}$ is bounded in $E$, thanks to Step 1, up to a
subsequence, there exists $u_0 \in E$ such that
\begin{equation} \label{e3.4}
\begin{gathered}
u_{b_n}\rightharpoonup u_0 \quad \text{in} E,
\\
u_{b_n}\to u_0 \quad \text{in $ L^p(\mathbb{R}^N)$ for } p\in[2,2^{**}),
\\
u_{b_n}\to u_0\quad \text{a.e.\  in}\mathbb{R}^N.
\end{gathered}
\end{equation}
Given that $\{u_{b_n}\}$ is a weak solution to \eqref{e1.1} with $b=b_n$, we have
\begin{equation} \label{e3.5}
\begin{aligned}
&\int_{\mathbb{R}^N} \left( \Delta u\Delta\phi+\nabla u\nabla \phi+V(x)u\phi
\right)\,dx +b_n \int_{\mathbb{R}^N}|\nabla u|^2dx \int_{\mathbb{R}^N}\nabla u\nabla
\phi\,dx
\\
&=\lambda\int_{\mathbb{R}^N}f(u)\phi\,dx
+\int_{\mathbb{R}^N}|u|^{2^{**}-2}u\phi\,dx
\end{aligned}
\end{equation}
for all $\phi \in C_0^\infty(\mathbb{R}^N)$.

Combing \eqref{e3.4}, \eqref{e3.5} and Step 1, we find that
\begin{equation} \label{e3.6}
\begin{aligned}
&\int_{\mathbb{R}^N} \left( \Delta u_0\Delta\phi+\nabla u_0\nabla
\phi+V(x)u_0\phi \right)\,dx +b_n \int_{\mathbb{R}^N}|\nabla u_0|^2dx
\int_{\mathbb{R}^N}\nabla u_0\nabla \phi\,dx
\\
&=\lambda\int_{\mathbb{R}^N}f(u_0)\phi\,dx
+\int_{\mathbb{R}^N}|u_0|^{2^{**}-2}u_0\phi\,dx
\end{aligned}
\end{equation}
for all $\phi \in C_0^\infty(\mathbb{R}^N)$, which in turn implies that $u_0$ is a weak solution to \eqref{e1.14}.
Analogous to the process of Lemma \ref{lem2.3}, we obtain that $u_0^\pm\neq0$. Thus, we have completed the proof of this step.
\smallskip

\noindent\textbf{Step 3.} In this step, we prove that problem \eqref{e1.14} possesses
a least energy nodal solution $v_0$, and that there exists a unique
pair $(\alpha_{b_n}, \beta_{b_n})\in\mathbb{R}^+\times\mathbb{R}^+$ satisfying
$\alpha_{b_n}v_0^++ \beta_{b_n}v_0^-\in\mathcal{N}_{b_n}^\lambda$.
Also we prove that $(\alpha_{b_n}, \beta_{b_n})\to (1, 1)$ as $n\to\infty$.

Similar to the proof of Theorem \ref{th1.1}, we can reach the
conclusion that problem \eqref{e1.14} possesses a least energy nodal
solution $v_0$, where $I_0^\lambda(v_0) = c_0$ and
$(I_0^\lambda)'(v_0) = 0$. Then, in view of Lemma \ref{lem2.1}, we
can obtain with ease the existence and uniqueness of the pair
$(\alpha_{b_n}, \beta_{b_n})$ such that
$\alpha_{b_n}v_0^++\beta_{b_n}v_0^-\in\mathcal{N}_{b_n}^\lambda$.
Besides, we know $\alpha_{b_n} > 0$ and $\beta_{b_n} > 0$.
To complete the proof, we just establish that $(\alpha_{b_n}, \beta_{b_n})\to(1, 1)$
as $n \to\infty$. Actually, given that
$\alpha_{b_n}v_0^++\beta_{b_n}v_0^- \in \mathcal{N}_{b_n}^\lambda$, we have
\begin{equation} \label{e3.7}
\begin{aligned}
&\alpha_{b_n}^2\|v_0^+\|^2
+\alpha_{b_n}\beta_{b_n}\int_{\mathbb{R}^N}\Delta v_0^+ \Delta v_0^-dx
\\
&+\alpha_{b_n}^2 b_n \int_{\mathbb{R}^N}|\nabla v_0^+|^2dx \Big( \alpha_{b_n}^2\int_{\mathbb{R}^N}|\nabla v_0^+|^2dx
+\beta_{b_n}^2 \int_{\mathbb{R}^N}|\nabla v_0^-|^2dx\Big)
\\
&= \lambda\int_{\mathbb{R}^N}f(\alpha_{b_n}v_0^+)\alpha_{b_n}v_0^+\,dx
+\int_{\mathbb{R}^N}|\alpha_{b_n}v_0^+|^{2^{**}}dx
\end{aligned}
\end{equation}
and
\begin{equation} \label{e3.8}
\begin{aligned}
&\beta_{b_n}^2\|v_0^-\|^2
+\alpha_{b_n}\beta_{b_n}\int_{\mathbb{R}^N}\Delta v_0^+ \Delta v_0^-dx
\\
&+\beta_{b_n}^2 b_n \int_{\mathbb{R}^N}|\nabla v_0^-|^2dx
\Big( \beta_{b_n}^2\int_{\mathbb{R}^N}|\nabla v_0^-|^2dx
+\alpha_{b_n}^2 \int_{\mathbb{R}^N}|\nabla v_0^+|^2dx\Big)
\\
&= \lambda\int_{\mathbb{R}^N}f(\beta_{b_n}v_0^-)\beta_{b_n}v_0^-\,dx
+\int_{\mathbb{R}^N}|\beta_{b_n}v_0^-|^{2^{**}}dx.
\end{aligned}
\end{equation}a
Since $b_n \searrow 0$, we conclude that the sequences
$\{\alpha_{b_n}\}$ and $\{\beta_{b_n}\}$ are bounded.
Assume, up to a subsequence, $\alpha_{b_n}\to\alpha_0$ and
$\beta_{b_n}\to\beta_0$. Then by \eqref{e3.7} and \eqref{e3.8}, we have
\begin{equation} \label{e3.9}
\alpha_0^2\|v_0^+\|^2 +\alpha_0\beta_0\int_{\mathbb{R}^N}\Delta v_0^+\Delta
v_0^-dx = \lambda\int_{\mathbb{R}^N}f(\alpha_0 v_0^+)\alpha_0 v_0^+\,dx
+\int_{\mathbb{R}^N}|\alpha_0 v_0^+|^{2^{**}}dx
\end{equation}
and
\begin{equation} \label{e3.10}
\beta_0^2\|v_0^+\|^2 +\alpha_0\beta_0\int_{\mathbb{R}^N}\Delta v_0^+\Delta
v_0^-dx = \lambda\int_{\mathbb{R}^N}f(\beta_0 v_0^-)\beta_0 v_0^-\,dx
+\int_{\mathbb{R}^N}|\beta_0 v_0^-|^{2^{**}}dx.
\end{equation}
Noticing that $v_0$ is a nodal solution to problem \eqref{e1.14}, we
obtain
\begin{gather}\label{e3.11}
\|v_0^+\|^2 +\int_{\mathbb{R}^N}\Delta v_0^+\Delta v_0^-dx =
\lambda\int_{\mathbb{R}^N}f(v_0^+)v_0^+\,dx +\int_{\mathbb{R}^N}| v_0^+|^{2^{**}}dx,
\\
\label{e3.12}
\|v_0^+\|^2 \int_{\mathbb{R}^N}\Delta v_0^+\Delta v_0^-dx =
\lambda\int_{\mathbb{R}^N}f(v_0^-)v_0^-\,dx +\int_{\mathbb{R}^N}|v_0^-|^{2^{**}}dx.
\end{gather}
Therefore, from \eqref{e3.9}-\eqref{e3.12}, we can easily obtain that
$(\alpha_{0}, \beta_{0}) = (1, 1)$, and thus Step 3 follows.

We  can now complete the proof of Theorem \ref{th1.3}.
We claim that $u_0$ obtained in
Step 2 is a least energy solution to problem \eqref{e1.14}.
In fact, according to Step 3 and Lemma \ref{lem2.1}, we see that
\begin{align*}
I_0^\lambda(v_0) \leq I_0^\lambda(u_0)
&=\lim_{n\to\infty}  I_{b_n}^\lambda(u_{b_n}) \\
&\leq
\lim_{n\to\infty}
I_{b_n}^\lambda(\alpha_{b_n}v_0^++\beta_{b_n}v_0^-) \\
&= \lim_{n\to\infty}  I_0^\lambda(v_0^++v_0^-) \\
&=I_0^\lambda(v_0),
\end{align*}
which yields completest the proof of Theorem \ref{th1.3}.
\end{proof}

\subsection*{Acknowledgments}

H. Pu was supported by the Graduate Scientific Research Project of
Changchun Normal University (SGSRPCNU No.\ 2020-51).
S. Liang was supported by the
Foundation for China Postdoctoral Science Foundation (Grant no.
2019M662220),  Scientific research projects for Department of
Education of Jilin Province, China (JJKH20210874KJ), Natural Science
Foundation of Changchun Normal University (No.\ 2017-09).
D. D. Repov\v{s} was supported by the Slovenian Research Agency (No.
P1-0292, N1-0114, N1-0083, N1-0064, and J1-8131).
We want to thank the anonymous referees for their comments and suggestions.

\end{document}